\documentclass[11pt]{amsart}
\usepackage{amssymb,amsmath,amsthm}
\newcommand{\leg}[2]{\genfrac{(}{)}{}{}{#1}{#2}}

\newtheorem{theorem}{Theorem}
\newtheorem{lemma}[theorem]{Lemma}
\newtheorem{corollary}[theorem]{Corollary}

\newtheorem{proposition}[theorem]{Proposition}
\newtheorem{definition}[theorem]{Definition}

\theoremstyle{remark}

\newtheorem*{remark}{Remark}

\numberwithin{theorem}{section} \numberwithin{equation}{section}

\newcommand{\rational}{\mathbb{Q}}

\newcommand{\Q}{\mathbb{Q}}

\newcommand{\Z}{\mathbb{Z}}
\newcommand{\N}{\mathbb{N}}
\newcommand{\Nrm}{\mathcal{N}}

\begin{document}
\title[Real Quadratic Fields and Hypergeometric Series]{Multiplicative $q$-hypergeometric series arising from real quadratic fields}

\author{Kathrin Bringmann and Ben Kane} 
\address{Mathematical Institute\\University of
Cologne\\ Weyertal 86-90 \\ 50931 Cologne \\Germany}
\email{kbringma@math.uni-koeln.de}
\address{Math Department\\
Radboud University\\
Postbus 9010\\
6500 GL, Nijmegen, Netherlands 
}
\email{bkane@science.ru.nl} 

\thanks{ The first author was partially supported by NSF grant DMS-0757907.}
\maketitle

\begin{abstract}
Andrews, Dyson, and Hickerson showed that 2 
$q$-hypergeometric series,   going back to Ramanujan, are   related to  real quadratic fields, which explains interesting properties of their Fourier  coefficients. There is also an interesting relation of such series to automorphic forms.
Here we construct more such examples arising from interesting combinatorial statistics.

\end{abstract}

\section{Introduction and statement of results} 
 Andrews, Dyson and Hickerson \cite{ADH}  intensively studied the function
\begin{equation} \label{sigma}
\sigma(q):=1+ \sum_{n=1}^{\infty}\frac{q^{\frac{n(n+1)}{2}} }{(-q)_n} 
= 1+q-q^2+2q^3+ \dots +4q^{45}+ \dots
\end{equation}
which 
first appeared in Ramanujan's lost notebook \cite{Ra}.
Here we define as usual 
$(a)_n=(a;q)_n:=\prod_{j=0}^{n-1}\left( 1-aq^j\right)$.  
In \cite{ADH}, the authors obtain an exact formula for the coefficients of (\ref{sigma}) by 
relating  this function to the arithmetic of $\Q(\sqrt{6})$. This implies that the coefficients have 
multiplicative properties, $\sigma$ is lacunary, i.e its coefficients are almost always zero, and 
yet attains every integer infinitely many times. 
In subsequent work a few more such examples were found  
by Lovejoy, Corson et al and others, building on known Bailey pairs, which led to interesting applications 
(see for example \cite{AJO,CFLZ, Lo, Lo2}).
Additional to these properties the function $\sigma$ has a natural combinatorial interpretation as the number of partitions into distinct parts with even rank minus those with odd rank. 
Recall that Dyson's rank of a partition \cite{Dyson} is defined as its largest part minus its number of parts.

The function $\sigma$ is in several ways related to automorphic forms.
Firstly, using  Hecke $L$-series,  Cohen \cite{Co} related  this function   to classical Maass forms.  
Secondly  if instead of partitions into distinct parts one considers unrestricted partitions one obtains Ramanujan's mock theta function 
\begin{equation*}
f(q):=1+ \sum_{n=1}^{\infty}\frac{q^{n^2}}{(-q)_n^2} 
\end{equation*}
which is, due to recent work of Zwegers \cite{Zw}  and Bringmann-Ono \cite{BO,BO2}, known to be the holomorphic part  of a harmonic Maass form,
 as are all the rank generating functions.
 Harmonic  Maass forms are generalizations of  modular forms, 
in that they   satisfy    the same transformation law, and (weak) growth conditions at 
cusps, but instead of being holomorphic, they are annihilated 
by the weight $k$ hyperbolic Laplacian.  
To  describe a third  way in which $\sigma$ is related to automorphic forms, 
we recall from \cite{ADH} that  the  main step in relating $\sigma$ to $\Q(\sqrt{6})$
 is to prove the following representation  of $\sigma$ as Hecke-type sums
$$
\sigma(q)= \sum_{\substack{ n \geq 0 \\ |j| \leq n} } (-1)^{n+j}\, q^{\frac{n(3n+1)}{2} -j^2}
\left(1-q^{2n+1} \right).
$$
We note that this representation can be also viewed as a ``false mock theta function". 
To be more precise,  
using work of Zwegers \cite{Zw},
the function 
$$
\sum_{\substack{ n \geq 0 \\ |j| \leq n} } (-1)^{n+j}\, q^{\frac{n(3n+1)}{2} -j^2}
\left(1+q^{2n+1} \right)
$$ 
can be viewed as the ``holomorphic part'' of an indefinite theta series and is related to one of Ramanujan' s 
sixth order mock theta functions.  

In this paper we find more examples resembling (\ref{sigma}) which have interesting combinatorial interpretations and which we relate to the real quadratic fields $\Q(\sqrt{2})$ and $\Q(\sqrt{3})$. 
Of particular interest are Theorems \ref{16Theorem} and \ref{16Companion}, since for its proof we find two new Bailey pairs (see Theorem \ref{BaileyTheorem}) which are   of  independent interest.

We first consider $q$-hypergeometric series related to  $K=\Q(\sqrt{2})$ and denote as usual with $\Nrm(\textbf{a})$ the norm of an ideal $\textbf{a}$ in $O_K$, the ring of integers of $K$.
\begin{theorem} \label{16Theorem}
The function
$$
f_1(q) :=\sum_{n=0}^{\infty} \frac{q^{ \frac{n^2+n}{2 }} (q)_{2n}}{(-q)_n(q)_{2n+1}} = 1+2q +3q^3+q^5+2q^6+2q^7+4q^{10} + \dots + 6q^{52} + \dots 
$$ 
 satisfies
\begin{equation} \label{2mod16eqn}
q  f_1(q^{16})=  
\sum_{   
\substack{
\textbf{a}\,  \subset O_{K} \\ \Nrm(\textbf{a}) \equiv  1 \pmod{16}
}} 
q^{\Nrm(\textbf{a})}.
\end{equation}
\end{theorem} 
\noindent
Let us next describe how the function $f_1(q)$ naturally arises in the theory of partitions.  For this,
 we let $P_1$ be the set of partitions with the following properties:  One is the only part which may be repeated, and if it is repeated, say occurring $r+1$ times, then the second smallest part (if it exists) is at least $2r$ and the third smallest part (if it exists) is at least two more than the second smallest part.  In Subsection \ref{combinatorics1}, we show that $f_1(q)-q$ is the generating function for partitions $\lambda \in P_1$ with $r_1(\lambda)$ even minus those with $r_1(\lambda)$ odd, where $r_1(\lambda)$ counts the largest non-repeated part minus the number of non-repeated parts if there is not exactly one non-repeated part, and $r_1(\lambda)$ equals the largest non-repeated part otherwise.

\begin{theorem} \label{16Companion}
The function 
$$
f_2(q) :=\sum_{n=1}^{\infty} \frac{q^{ \frac{n^2+n}{2}} (q)_{2n-2}}{(-q)_{n-1} (q)_{2n-1}} = q+q^2 + 2q^3+2q^5+2q^6+\dots + 4q^{14} + \dots + 3q^{77}+\dots
$$
satisfies
\begin{equation} \label{2mod16comp}
q^{-7}  f_2(q^{16})=  
\sum_{
\substack{
\textbf{a}\,  \subset O_K \\ \Nrm(\textbf{a}) \equiv 9 \pmod{16}
 }} q^{\Nrm(\textbf{a})}.
\end{equation}
\end{theorem}
\noindent
The partition theoretic interpretation of $f_2(q)$ has striking similarities to that of $\sigma(q)$.  
Define $P_2$ to be the set of partitions into distinct parts restricted by the following conditions:  
 The rank  of $\lambda \in P_2$ is at least $2(s-1)$, where $s$ is the smallest part, 
 and the second smallest part is at least $2s$.  
 In Subsection \ref{combinatorics2}, we show that  $f_2(q)$
  is the generating function for $\lambda\in P_2$ with $r_2(\lambda)$ even minus those with $r_2(\lambda)$ odd, where $r_2(\lambda)$ is the rank if there is more than one part and $r_2(\lambda)=1$ otherwise.

\begin{theorem}\label{2twistthm}
The function
$$
f_3(q):=\sum_{n=0}^{\infty} 
 \frac{(q)_{2n}}{(-q)_{2n+1}} q^n = 1-2q^3 +q^4+2q^8 -2q^{11} + \dots + 3q^{24} + \dots
$$ 
satisfies 
\begin{equation}\label{2twisteqn}
qf_3(q^2)= 
\sum_{    
\textbf{a}\,  \subset O_K } 
\leg{-4}{\Nrm(\textbf{a})}
q^{\Nrm(\textbf{a})},
\end{equation}
where $\leg{c}{d}$ denotes the Kronecker symbol. 
\end{theorem} 
\noindent To relate $f_3$ to partitions, we  
recall the notion of overpartitions and overpartition pairs. 
An \textit{overpartition}   \cite{CL} of n is a partition of n in which the first occurrence of a number can be 
overlined.
An \textit{overpartition pair} \cite{Lo3} of n is a pair of overpartitions $(\mu, \lambda)$, 
where the sum of all the parts is n. 
Here we  consider
$P_3$ as  the set of  overpartitions  pairs $\Lambda=(\mu,\lambda)$  with the following restrictions on $\mu$ and $\lambda$:  If the largest part of $\mu$ is overlined then it must also occur non-overlined and if $\mu$ is the empty partition then no parts of $\lambda$ may be overlined.  Moreover, if $\lambda$ has any   parts, then its largest  part is exactly one greater than the largest  part of $\mu$, parts of $\lambda$   of size one less than its largest part cannot be overlined, and the number of occurrences of the largest  part in $\lambda$ is at least half of all parts in $\lambda$ (with strict inequality if the largest part of $\lambda$ is overlined).
In Subsection \ref{combinatorics3} we show that $f_3(q)$ is the generating function for $\Lambda\in P_3$ with $r_3(\Lambda)$ odd minus those with $r_3(\Lambda)$ even, where $r_3(\Lambda)$ counts the number of occurrences of the largest part in $\lambda$  minus the number of   parts in $\mu$.
\begin{theorem}\label{2mod1thm}
The function  
$$
f_4(q):=\sum_{n=0}^{\infty} 
 \frac{(q)_{2n+1}}{(-q)_{2n+2}} q^{n+1} = q -q^2 -q^4 +2q^7 -q^8 +q^9 -2q^{14} +\dots - 3q^{98} + \dots 
$$ 
satisfies 
\begin{equation}\label{2mod1eqn}
f_4(q)= -
\sum_{ 
\textbf{a}\,  \subset O_K } 
  (-1)^{\Nrm(\textbf{a})} 
 q^{\Nrm(\textbf{a})}.
\end{equation} 
\end{theorem}
\noindent
The combinatorial interpretation of $f_4$ is similar to that of $f_3$. 
We let $P_4$ be the set of  overpartition pairs $\Lambda=(\mu,\lambda)$ 
 with the following restrictions on $\mu$ and $\lambda$: If the largest part of $\mu$ is overlined then it must also occur non-overlined.  Moreover, if $\lambda$ has any parts, then its  largest part equals the largest  part of $\mu$, 
   the largest part of $\lambda$  cannot be overlined,  $\lambda$  has an even number of parts, and the number of occurrences of the largest  part in $\lambda$ is at least half of the total number of its parts.
In Subsection \ref{combinatorics4} we show that  $f_4(q)$ is the generating function for $\Lambda\in P_4$ with $r_4(\lambda)$ odd minus those with $r_4(\lambda)$ even, where $r_4(\lambda)$ is half the number of parts of $\lambda$ minus the number of parts of $\mu$.  

We next turn to the  real quadratic field $L:=\Q(\sqrt{3})$.
\begin{theorem}\label{3mod4thm}
The function 
$$
f_5(q):=\sum_{n=0}^{\infty} 
 (-1)^n \frac{(q)_n}{(q;q^2)_{n+1}} q^{\frac{n^2+n}{2}} = 1 + q^2 + 2q^3 +q^6 +2q^8  + \dots + 3q^{42} + \dots
$$ 
satisfies
\begin{equation}\label{3mod4eqn}
qf_5(q^4) = 
\sum_{   
\substack{
\textbf{a}\,  \subset O_L \\ \Nrm(\textbf{a}) \equiv  1 \pmod{4}
}}   
q^{\Nrm(\textbf{a})}.
\end{equation}
\end{theorem}
\noindent
To see how   $f_5$ can be viewed in the framework of partitions,  we require some notation.
For a partition $\lambda=(\lambda_1,\dots,\lambda_n)$ into $n$ distinct parts (in decreasing order), define the sequence 
$\ell_i:=\lambda_i-\lambda_{i+1}$ ($1\leq i \leq n-1$) and  $\ell_n:=\lambda_n$. 
   Also, define $E:=E_{\lambda}:=\{ 2 \leq  r \leq n: \ell_r\text{ is even}\}$ and $e:=\# E$.  
 Let $P_5$ be the set of partitions $\lambda$ into distinct parts with the following properties:  
 If  $\lambda$ has only   a single part, then this part is not congruent to 1 modulo 3, and otherwise $\ell_1\geq d_{\lambda}+1$ and $\ell_1\equiv d_{\lambda}+1\pmod{3}$, where $d_{\lambda}$ is defined in Subsection \ref{combinatorics5}.
In Subsection \ref{combinatorics5}, we show that   $f_5(q)$ 
 is the generating function for partitions $\lambda\in P_5$ with $r_5(\lambda)$ odd minus those with $r_5(\lambda)$ even, where $r_5(\lambda):=\lambda_2$, if the partitions contains at least two parts and  
 $r_5(\lambda):=1$ otherwise.  
 \begin{remark}
 We note that  the $q$-hypergeometric series  
 $$
\sum_{n=0}^{\infty} 
\frac{(q)_n}{(q;q^2)_{n+1}} q^{\frac{n^2+n}{2}} = 1 + 2q + q^2 +2q^3 + 2q^4 + 3q^6 + \dots +6q^{81}+\dots ,
$$
which is obtained from $f_5(q)$ by deleting the $(-1)^n$, 
 is the  modular form   
 $\frac{\eta^4(2 \tau)}{\eta^2(\tau)}$
 which is related to  
$\rational(i)$.  
\end{remark}
\begin{theorem}\label{3mod4compthm}
The function 
$$
f_6(q):=\sum_{n=1}^{\infty} 
 (-1)^n \frac{(q)_{n-1} (q^2;q^2)_{n-1}}{(q)_{2n-1}}q^n = -q-2q^3 -2q^6-q^7-2q^{10} -\dots -4q^{36}- \dots 
$$
satisfies 
\begin{equation}\label{3mod4compeqn}
q^{-1} f_6(q^4)= -
 \sum_{   
\substack{
\textbf{a}\,  \subset O_L \\ \Nrm(\textbf{a}) \equiv  3 \pmod{4}
}} 
 q^{\Nrm(\textbf{a})}.
\end{equation}
\end{theorem} 
\noindent
Let $P_6$ be the set of overpartitions $\lambda$ with the following properties:  the largest part cannot be overlined, every overlined part must also occur non-overlined, and the number of repetitions of the largest part plus the number of overlined parts is greater than half of all of the parts. 
In Subsection \ref{combinatorics6}, we show that  
 $f_6(q)$ is the generating function for $\lambda\in P_6$ with $r_6(\lambda)$ even minus those with $r_6(\lambda)$ odd, where $r_6(\lambda)$ counts the largest part minus the number of overlined parts.

\begin{theorem} \label{3case}
The function
$$
f_7(q):=
\sum_{n =0}^{\infty} 
 (-1)^n\, \frac{q^{n^2+n}\left(q^2;q^2\right) _n}{(-q)_{2n+1} } = 1-q +2q^4-q^5-2q^7 +q^8 +\dots + 3q^{40} + \dots
$$
satisfies  
\begin{equation} \label{3caseeqn}
qf_7(q^3)= 
 -  \sum_{   
\substack{
\textbf{a}\,  \subset O_L\\ \Nrm(\textbf{a}) \equiv  1 \pmod{3}
}}   
(-1)^{\mathcal{N}(\textbf{a})}\, q^{\Nrm(\mathbf{a})}.
\end{equation}
\end{theorem}
\noindent
To interpret $f_7$, we denote for  an overpartition $\lambda$ by  $M(\lambda)$  the number of times that the largest part occurs if the largest part is greater than one, and set $M(\lambda)=0$ otherwise.  
Let $P_7$ be the set of overpartitions with the following properties:  If the largest part is greater than one then the largest part equals the number of non-overlined parts plus one, the second largest part size is at most $M(\lambda)+1$, and only parts of size less than or equal to $M(\lambda)$ may be overlined.
In Subsection \ref{combinatorics7} we show that  
 $f_7(q)$ is the generating function for $\lambda\in P_7$  with an even number of parts minus those with an  odd  number of parts. 
\begin{theorem}\label{3mod3compthm}
The function
$$
f_8(q):=\sum_{n=1}^{\infty} \frac{(q)_{2n-1}}{(q^{2n};q^2)_{n}}q^n = q +q^3 -2q^4 -2q^8 +2q^9 + \dots -4q^{48} + \dots +3q^{81} + \dots
$$ 
satisfies 
\begin{equation}\label{3mod3compeqn}
q^{-1} f_8(q^3)= 
\sum_{   
\substack{
\textbf{a}\,  \subset O_L \\ \Nrm(\textbf{a}) \equiv  2 \pmod{3}
}}  
(-1)^{\Nrm(\textbf{a})}
q^{\Nrm(\textbf{a})}.
\end{equation}
\end{theorem} 
\noindent
Let $P_8$ be the set of overpartitions with the following properties:  The largest part cannot be overlined and the number of repetitions of the largest part size is greater than half of all non-overlined parts.  In Subsection \ref{combinatorics8}, we show that   $f_8$ is the generating function for overpartitions $\lambda\in P_8$ with an odd number of parts minus those with an even number of parts.




We next use the arithmetic of $K$ and $L$ to determine properties for the functions $f_i$ which resemble those of $\sigma$. 
\begin{corollary} \label{LacCorollary}
The functions  $f_i$ are lacunary. 
\end{corollary} 
We let 
\begin{eqnarray*}
S_i&:=&
\left\{  
m \in \Z, \text{ there are infinitely many } n \text{ such that } a_{f_i}(n)=m, \right. \\
& & \left.
\text{where }  a_{f_i}(n)
\text{ denotes the }  n \text{ th coefficient of } f_i.
\right\}
\end{eqnarray*}
\begin{corollary} \label{InfCorollary}
We have  
$$
S_i= 
\left\{ 
\begin{array}{ll}
\N_0&\text{if  } i \in \{1,2,5\},\\
-\N_0&\text{if } i=6,\\ 
\Z&\text{if } i \in \{ 4,7\},\\ 
\N_0 \cup -2 \N&\text{if } i\in \{3,8 \}.
\end{array}
\right.
$$  
\end{corollary} 

It would be interesting to further investigate the functions $f_i$ recovered here, for example relating them
 to   harmonic Maass forms as described above for $\sigma$. 
In particular the combinatorics of $f_2$ resemble those of $\sigma$. 
So the question  arises, whether one can construct new mock theta functions by  
 considering unrestricted partitions instead of partitions into distinct parts.
Furthermore $\sigma$ also occurred in interesting number theoretical identities involving sums of tails of $\eta$-quotients (see for example \cite{AJO,LO}).
I would be interesting to investigate whether our functions $f_i$ play related roles. 
We plan to address these questions in future research.

The paper is organized as follows.  
Section \ref{NewBailey} is devoted to establishing two new Bailey pairs required for the proofs of Theorems \ref{16Theorem} and \ref{16Companion} which are of independent interest. 
In Section \ref{ProofSection} we write the functions $f_i$ as Hecke-type sums and establish our main theorems. 
In Section \ref{combinatorics} we prove  the natural connection of the functions $f_i$ to the above  described partition statistics. 
\section*{Acknowledgements}
The authors thank Jeremy Lovejoy  for helpful comments on earlier versions of this paper.

\section{Two new Bailey pairs} \label{NewBailey}
Here we establish  two new Bailey pairs required for the proofs of Theorems  \ref{16Theorem} and 
\ref{16Companion}. 
Let us first recall the definition of a Bailey pair and Bailey's Lemma.  
For details and background on Bailey pairs, we refer the reader to Chapter 3 of \cite{An2}.  
\begin{definition}
Two sequences $\{\alpha_n\}$ and $\{\beta_n\}$  form a \textbf{Bailey pair relative to a} 
 if 
for all $n \geq 0$, we have
$$
\beta_n= \sum_{r=0}^n
\frac{\alpha_r}{(q)_{n-r}(aq)_{n+r}}.
$$ 
\end{definition}
Moreover if only the $\beta_n$ are given, then $\alpha_n$ can be determined  using Bailey inversion:
\begin{equation} \label{BaileyInversion}
\alpha_n= \left( 1-aq^{2n}\right) 
\sum_{j=0}^n \frac{(aq)_{n+j-1}(-1)^{n-j}q^{n-j \choose 2} }{(q)_{n-j}} 
\beta_j.
\end{equation} 

To establish our main theorems we will use a limiting case of Bailey's Lemma. 
\begin{lemma} \label{BaileyLemma}
If $\alpha_n$ and $\beta_n$ form a Bailey pair relative to $a$, then 
we have, providing both side converge absolutely,
$$
\sum_{n=0}^{\infty} 
\frac{(\rho_1,\rho_2)_n}{ \left( \frac{aq}{\rho_1},\frac{aq}{\rho_2} \right)_n} 
\left(\frac{aq}{ \rho_1 \rho_2} \right)^n \, \alpha_n
= \frac{\left(aq,\frac{aq}{\rho_1 \rho_2}   \right)_{\infty}}{ \left(\frac{aq}{\rho_1} ,\frac{aq}{\rho_2}   \right)_{\infty}} 
\sum_{n=0}^{\infty}(\rho_1,\rho_2)_n \left(\frac{aq}{\rho_1 \rho_2}  \right)^n \, \beta_n,
$$
where $(a_1, \dots,a_r;q)_n=(a_1,\dots,a_r)_n:=\prod_{j=1}^{r}(a_j)_n.$
\end{lemma}

We show the following.
\begin{theorem} \label{BaileyTheorem}
\begin{enumerate}
\item  
The sequences $a_n,b_n$ form a Bailey pair relative to  $a=1$:
\begin{eqnarray*}
b_0=b_0(q)&:=&0,\\
b_n=b_n(q)&:=&\frac{(-1)^n (q;q^2)_{n-1}}{(q)_{2n-1}},\\   
a_{2n} =a_{2n}(q)&:=& (1-q^{4n}) q^{2n^2-2n}\sum_{j=-n}^{n-1} q^{-2j^2-2j},\\
a_{2n+1}=a_{2n+1}(q) &:=& -(1-q^{4n+2}) q^{2n^2}\sum_{|j|\leq n} q^{-2j^2}.
\end{eqnarray*}
\item The sequences $\alpha_n,\beta_n$ form a Bailey pair relative to  $a=q$:
   \begin{eqnarray*}
   \beta_n=\beta_n(q)&:=&\frac{(-1)^n\left(q;q^2 \right)_n}{(q)_{2n+1}},\\  
      \alpha_{2n}=\alpha_{2n}(q)&:=& \frac{1}{1-q} \left(q^{2n^2+2n} \sum_{-n\leq j\leq n-1} q^{-2j^2-2j}  + q^{ 2n^2} \sum_{|j| \leq n} q^{-2j^2}\right),\\
   \alpha_{2n+1}=\alpha_{2n+1}(q)&:=&- \frac{1}{1-q} \left(q^{2n^2+4n+2} \sum_{|j| \leq n}q^{-2j^2}   
   + q^{ 2n^2+2n}
   \sum_{-n-1\leq j \leq n} q^{-2j^2-2j} \right).
   \end{eqnarray*}
\end{enumerate}
\end{theorem}
\begin{proof}

We first define 
\begin{equation} \label{DefineUn}
   U_n:=(-1)^n 
  \sum_{j=1}^n   \left[
  \begin{matrix}n+j-1\\2j-1\end{matrix}
   \right] 
 q^{{n-j} \choose {2}}  (q;q^2)_{j-1} ,
   \end{equation}
   where $\left[\begin{matrix}n\\j \end{matrix}  \right]:=\frac{(q)_n}{(q)_j(q)_{n-j}}.$
   \begin{proposition} \label{Un}
   We have for $n \geq 0$:
   \begin{eqnarray*}
   U_{2n}&=&q^{2n^2-2n}  \sum_{j=-n}^{n-1} q^{-2j^2-2j},\\ 
   U_{2n+1}&=&-q^{2n^2}\sum_{|j|\leq n}q^{-2j^2}
   .
   \end{eqnarray*}
   \end{proposition}
   \begin{proof}
   It is easy to see that both sides of Proposition \ref{Un} satisfy  $U_0=0$ and  $U_1= -1$.
   To finish the proof it is enough to show that (\ref{DefineUn}) satisfies  the recurrence
   \begin{equation}  \label{Urelation}
   U_{n+2}= q^{2n}U_n+2(-1)^n.
   \end{equation}
   To prove (\ref{Urelation}), we let
     $$
   V_n:=
   (-1)^n \left( U_{n+2}-q^{2n}U_n\right).
   $$
   Inserting (\ref{DefineUn})   yields 
\begin{eqnarray*}
   V_n&=& \sum_{j=1}^{n+2} \frac{(q)_{n+1+j}q^{\frac12(n+2-j)(n+1-j)} }{(q)_{2j-1}(q)_{n-j+2}} 
 (q;q^2)_{j-1}
 - \sum_{j=1}^{n} \frac{(q)_{n-1+j}q^{\frac12(n-j)(n-1-j)+2n} }{(q)_{2j-1}(q)_{n-j}} 
 (q;q^2)_{j-1}\\
&=&    q\sum_{j=1}^{n+2} 
   \frac{(q)_{n+j-1}(q;q^2)_j}{(q)_{2j-1}(q)_{n-j+2}} 
   \left( 
   1-q^{2n+2}
   \right)q^{\frac12(n-j)^2+\frac{3n}{2}-\frac{3j}{2}  }.
\end{eqnarray*}
Using 
   \begin{eqnarray*}
   (q)_{n+j-1}&=& (q)_{n-1}\left( q^n \right)_j,\\
   (q)_{n+2-j}&=&
   \frac{(q)_{n+2}}{q^{ (n+2)j} q^{ -\frac{j(j-1)}{2}  } (-1)^j\left(q^{-n-2} \right)_j}, 
  \end{eqnarray*}
we obtain 
$$
   V_n =
    \frac{q^{1+\frac{n^2}{2}+\frac{3n}{2}} \left( 1+q^{n+1}\right)}{\left( 1-q^n\right) \left(1-q^{n+2}\right)}  
  \sum_{j=1}^{\infty} 
   \frac{(-1)^j\left(q^n,q^{-n-2} \right)_jq^j}{(-q,q)_{j-1}}.
$$
   We next recall the   following transformation which is due  to Heine (see (2.6) of \cite{An2}) 
   \begin{equation} \label{Heine}
   \sum_{n=0}^{\infty} \frac{(a,b)_n}{ (q,c)_n} \left(\frac{c}{ab} \right)^n
   = \frac{\left(\frac{c}{a} ,\frac{c}{b} \right)_{\infty}}{\left( c,\frac{c}{ab}\right)_{\infty}} .
   \end{equation}
   We use (\ref{Heine}) 
    with $a=q^{n+1}$, $b=q^{-n-1}$, and $c=-q$, yielding
   $$
   V_n=
   \frac{q^{\frac{n^2}{2}+\frac{n}{2} }\left(-q^{-n},-q^{n+2}\right)_{\infty} \left(1+q^{n+1} \right)}{(-q,-q)_{\infty}}.
   $$
Then $V_n=2$ follows from the identities 
   \begin{eqnarray*}
   \left( -q^{-n}\right)_{\infty}&=&2 q^{ -\frac{n(n+1)}{2} } (-q)_n(-q)_{\infty},\\
   \left( -q^{n+2}\right)_{\infty}&=&\frac{(-q)_{\infty}}{(-q)_{n+1}} .
   \end{eqnarray*}
This proves (\ref{Urelation}) and thus Proposition \ref{Un}. 
   \end{proof}  
   We now show  Theorem \ref{BaileyTheorem} (1). 
Using the Bailey inversion (\ref{BaileyInversion}) with $a=1$, we have 
\begin{eqnarray*}
a_n&=& (1-q^{2n})\sum_{j=0}^{n} \frac{(q)_{n+j-1}}{(q)_{n-j}} (-1)^{n-j}q^{\binom{n-j}{2}} b_j\\
&=& (1-q^{2n})\sum_{j=1}^{n} \frac{(q)_{n+j-1}}{(q)_{n-j}(q)_{2j-1}}(q;q^2)_{j-1} (-1)^{n}q^{\binom{n-j}{2}}\\
&=& (1-q^{2n}) (-1)^{n} \sum_{j=1}^{n} \left[
  \begin{matrix}n+j-1\\2j-1\end{matrix}
   \right]  (q;q^2)_{j-1}q^{\binom{n-j}{2}}\\
&=& (1-q^{2n})U_n.
\end{eqnarray*}
This directly gives (1) using Proposition  \ref{Un}.

We next turn to the proof of Theorem \ref{BaileyTheorem} (2) and  use   (\ref{BaileyInversion}) with $a=q$ to obtain 
  \begin{eqnarray*}
  \alpha_n&=&\frac{1-q^{2n+1}}{1-q} \sum_{j=0}^n \frac{(q)_{n+j}}{(q)_{n-j}} (-1)^{n+j}\, q^{{n-j} \choose{2}} \beta_j \\
  &=& \frac{1-q^{2n+1}}{1-q} (-1)^n\sum_{j=0}^n \frac{(q)_{n+j}(q;q^2)_j}{(q)_{n-j}(q)_{2j+1}}  q^{{n-j} \choose{2}}  \\
  &=& \frac{1}{1-q} (-1)^n\sum_{j=0}^n   \left[
  \begin{matrix}n+j+1\\n-j \end{matrix}
   \right]
  \frac{1-q^{2n+1}}{1-q^{n+j+1}}  q^{{n-j} \choose{2}}  (q;q^2)_j.
  \end{eqnarray*}
Writing 
  $$
  1-q^{2n+1}= 1-q^{n+j+1}+q^{n+j+1}\left( 1-q^{n-j}\right)
  $$
yields
  \begin{eqnarray*}
  \alpha_n &= &
  \frac{1}{1-q} (-1)^n \left(
  \sum_{j=0}^n   \left[
  \begin{matrix}n+j+1\\n-j \end{matrix}
   \right]  q^{{n-j} \choose{2}}  (q;q^2)_j 
   +   \sum_{j=0}^{n-1}   \left[
  \begin{matrix}n+j\\n-j-1 \end{matrix}
   \right]  q^{{n-j} \choose{2}}  q^{ n+j+1} (q;q^2)_j 
   \right) \\
   &=&\frac{1}{1-q} \left(-U_{n+1}+q^{2n}U_n \right).
   \end{eqnarray*}
Substituting $U_n$ from Proposition \ref{Un} gives the desired equality for $\alpha_n$.
\end{proof}
\section{Proofs of the Main Theorems} \label{ProofSection}

\begin{proof}[Proof of Theorem \ref{16Theorem}]

 The main step in the proof of Theorem \ref{16Theorem} is to rewrite $f_1$ as a Hecke-type sum, which relies on the  Bailey pair obtained in Theorem \ref{BaileyTheorem} (2).
  \begin{proposition}  \label{Hecke1}
We have  
$$
  f_1(q)= \sum_{ \substack{n \geq 0 \\ -n-1 \leq j \leq n  }} 
  q^{4n^2+5n+1-2 j^2-2j} \left(1+q^{ 6n+6} \right)
  + \sum_{ \substack{n \geq 0 \\ -n \leq j \leq n  }} 
  q^{4n^2+n-2 j^2} \left(1+q^{ 6n+3} \right).
  $$
  \end{proposition}
  \begin{proof}
  We use the Bailey pair from Theorem  \ref{BaileyTheorem} (2)   in 
  Bailey's Lemma with with $\rho_1 \to \infty$ and $\rho_2=q$.
  Using the fact  that 
  $$
  \lim_{\rho \to \infty} \frac{(\rho)_n}{\rho^n} 
  = (-1)^n\, q^{\frac{n(n+1)}{2}},
  $$
the ``$\beta$-side'' of Bailey's Lemma equals
   $$
\frac{1}{1-q} \sum_{n=0}^{\infty}  \frac{
   (q;q^2)_n (q)_n 
    q^{\frac{n(n+1)}{2} }
   }{ (q)_{2n+1} }
  =\frac{1}{1-q} \sum_{n=0}^{\infty} \frac{(q)_{2n}}{(-q)_n(q)_{2n+1}} \, q^{\frac{n(n+1)}{2} } 
  = \frac{1}{1-q}f_1(q).
   $$
The ``$\alpha$-side'' equals
$$
   \sum_{n =0}^{\infty} 
    q^{n(2n+1)}\alpha_{2n} 
   - \sum_{n=0}^{\infty} 
    q^{(2n+1)(n+1) } \alpha_{2n+1}.
$$
Plugging in $\alpha_n$ from Theorem \ref{BaileyTheorem} (2)  and multiplying by $1-q$ yields Proposition \ref{Hecke1}.
  \end{proof}
  To finish the proof of  Theorem \ref{16Theorem}, we first observe that 
  Propositon \ref{Hecke1} yields 
\begin{multline}\label{compsq1}
qf_1(q^{16}) = \sum_{\substack{n\geq 0\\ -n-1\leq j \leq n}} 
\left(q^{(8n+5)^2 - 2(4j+2)^2} + q^{(8n+11)^2 - 2(4j+2)^2} \right) \\
+ \sum_{\substack{n\geq 0\\ |j| \leq n}} \left(q^{(8n+1)^2 - 2(4j)^2} + q^{(8n+7)^2 - 2(4j)^2}\right).
\end{multline}
We next  use Lemma 3 of \cite{ADH}  and unique factorization in $O_K$ to rewrite each ideal $\mathbf{a}$ occurring in equation (\ref{2mod16eqn}) uniquely as 
 $\mathbf{a}=(u+v\sqrt{2})$  with $u> 0$ and $-\frac{1}{2} u < v\leq \frac{1}{2} u$.  
The congruence condition $\Nrm(\mathbf{a})= u^2-2v^2\equiv 1\pmod{16}$ corresponds to the four summands occurring in equation (\ref{compsq1}).  This completes the proof of Theorem \ref{16Theorem}.
  \end{proof} 
  \begin{proof}[Proof of Theorem \ref{16Companion}]  
\begin{proposition}\label{Hecke2}
We have 
$$
f_2(q) =  \sum_{\substack{n\geq 1 \\ -n \leq j \leq n-1  } } \left(1+q^{2n}\right) q^{4n^2-n-2j^2-2j} 
 + \sum_{\substack{n\geq 0\\ |j|\leq n}} q^{4n^2+3n+1-2j^2}\left(1+q^{2n+1}\right).
$$
\end{proposition} 
\begin{proof}  
We use  the Bailey pair from Theorem  \ref{BaileyTheorem}  (1)
in  Bailey's Lemma with     $\rho_1\to\infty$.  
We divide both sides by $1-\rho_2$,  and then let $\rho_2\to 1$. 
After this  the  ``$\beta$-side''   equals 
\begin{equation*}
\sum_{n\geq 1} \frac{(q)_{n-1}(q;q^2)_{n-1}}{(q)_{2n-1}}q^{\frac{n^2+n}{2}} =   \sum_{n\geq 1} \frac{(q)_{2n-2}}{(-q)_{n-1}(q)_{2n-1}}q^{\frac{n^2+n}{2}}=f_2(q). 
\end{equation*}
On the ``$\alpha$-side'', we have 
$$ 
\sum_{n= 1}^{\infty} \frac{(-1)^n q^{\frac{n^2+n}{2}}  }{1-q^n}   a_n   
=   \sum_{n=1}^{\infty}
 q^{2n^2+n}(1+q^{2n}) U_{2n} - \sum_{n=0}^{\infty}
  q^{2n^2+3n+1}(1+q^{2n+1})U_{2n+1}.  
$$
Inserting Proposition \ref{Un} proves Proposition \ref{Hecke2}.
\end{proof} 
Proposition \ref{Hecke2} gives  
\begin{multline} \label{Hecke2new}
q^{-7}f_2 \left(q^{16} \right) 
= \sum_{\substack{n\geq 1\\ -n\leq j\leq n-1}}  \left(q^{(8n-1)^2-2(4j+2)^2} + q^{(8n+1)^2 -2(4j+2)^2}\right)   \\
+ \sum_{\substack{n\geq 0\\ |j|\leq n}} \left(q^{(8n+3)^2-2(4j)^2} + q^{(8n+5)^2-2(4j)^2}\right).
\end{multline} 
We write $\textbf{a}$ in   (\ref{2mod16comp}) as in the proof of Theorem  \ref{16Theorem}. 
  The condition $\mathcal{N}(\textbf{a}) \equiv 9 \pmod{16}$ 
   translates in the four  summands of  
  (\ref{Hecke2new}), proving  Theorem \ref{16Companion}.


\end{proof}

\begin{proof}[Proof of  Theorem \ref{2twistthm}]
\begin{proposition} 
We have
\begin{equation}\label{fHecke}
f_3(q)=\sum_{\substack{n\geq 0\\ |j|\leq n}} (-1)^j q^{2n^2+2n-j^2}.
\end{equation}
 \end{proposition}
\begin{proof}
For the proof of Proposition \ref{fHecke} we recall the following Bailey pair  from Theorem 2.3 of \cite{AH}.  
\begin{theorem} \label{AndrewBailey}
The following $A_n'$ and $B_n'$ form a Bailey pair relative to $a$:
$$
A_n'= A_n'(q):=
\frac{q^{n^2} (bc)^n\left(1-aq^{2n} \right)\left( a/b,a/c\right)_n}{(1-a)(bq,cq)_n}
\sum_{j=0}^n\frac{(-1)^j\, \left( 1-aq^{2j-1}\right)(a)_{j-1}(b,c)_j}{q^{j(j-1)/2} (bc)^j(q,a/b,a/c)_j} 
$$
and 
$$
B_n'=B_n'(q):=\frac{1}{(bq,cq)_n}.
$$
\end{theorem}
We use the Bailey pair from Theorem \ref{AndrewBailey} 
with 
$q\to q^2$, $b=-q$, $c=-1$, and $a=q^2$ in  Bailey's Lemma with $\rho_1=q$ and  $\rho_2=q^2$.  
The ``$\beta$-side''   equals
$$
\frac{1}{1+q}
\sum_{n=0}^{\infty} \frac{(q,q^2;q^2)_n q^n}{(-q^3,-q^2;q^2)_n}       
=  \sum_{n=0}^{\infty} \frac{(q)_{2n}}{(-q)_{2n+1}}q^n=f_3(q).
$$  
The ``$\alpha$-side'' is
$$
\sum_{n =0}^{\infty}  q^{ 2n^2+2n} 
\left(1+2\sum_{j=1}^n (-1)^jq^{-j^2} \right),
$$
which is the right-hand side of  (\ref{fHecke}).  
\end{proof}
Using Proposition \ref{fHecke}, we get  
\begin{equation*}
qf_3(q^2)  
= \sum_{\substack{n\geq 0\\ |j|\leq n}} (-1)^j q^{(2n+1)^2-2j^2}.
\end{equation*}
We write each ideal $\mathbf{a}$ as 
$\mathbf{a}=(u+v\sqrt{2})$  as in the proof of Theorem \ref{16Theorem}.
Note that the Kronecker symbol implies that only elements with $u$ odd survive 
and that  
$(-1)^v = \leg{-4}{\Nrm (\textbf{a})}$.   
This finishes the proof of Theorem  \ref{2twistthm}.
\end{proof}
\begin{proof}[Proof of Theorem \ref{2mod1thm}]
\begin{proposition} \label{Hecke4}
We have 
\begin{equation*}
f_4(q)= -
\sum_{\substack{n\geq 0 \\ -n-1 \leq j \leq n  }} (-1)^j 
q^{2n^2+4n+2 -j^2}  .
\end{equation*}
\end{proposition}
\begin{proof}
We use  the Bailey pair from 
 Theorem \ref{AndrewBailey} with  
$q\to q^2$, $b=-q$, $c=-q^2$, and $a=q^4$ in Bailey's Lemma with $\rho_1=q^2$ and  $\rho_2=q^3$.  
The ``$\beta$-side''  then equals 
\begin{equation*}
 \frac{1-q}{1-q^4}
 \sum_{n=0}^{\infty} 
  \frac{(q^2,q^3;q^2)_n}{(-q^3,-q^4;q^2)_n}q^n 
 =   
    \frac{1}{1-q}\sum_{n=0}^{\infty}  \frac{(q)_{2n+1}}{(-q)_{2n+2}}q^n =\frac{q^{-1}}{1-q} f_4(q).
\end{equation*}
The ``$\alpha$-side" equals    
\begin{equation*}  
\sum_{n=0}^{\infty}  q^{2n^2+4n }  \left( 1+ \frac{q}{1-q}\sum_{j=1}^{n} (-1)^j (q^{-j^2-2j-1}-q^{-j^2}) \right) 
 = - \frac{q}{1-q}  
 \sum_{n=0}^{\infty}  q^{2n^2+4n }
 \sum_{j=-n-1}^{n} (-1)^{j} q^{-j^2}.
\end{equation*}
Multiplying with $q(1-q)$ gives Proposition \ref{Hecke4}.  
\end{proof}
By Proposition \ref{Hecke4}, we have 
$$
-f_4(q^2)  
= 
\sum_{\substack{n\geq 0\\-n-1 \leq j \leq n  }}  (-1)^jq^{(2n+2)^2-2j^2}.
$$
Using Lemma 3 of \cite{ADH},  we
  write each ideal as $\textbf{a}=(u + \sqrt{2}v)$ with $\Nrm(\textbf{a})=2v^2-u^2$ with $v>0$ and $-v<u \leq v$ .
As before this yields  
\begin{equation*}
f_4(q^2)=-
\sum_{\substack{  \textbf{a} \subset O_K \\ 2| \mathcal{N}(\textbf{a})   }} 
(-1)^{\frac{\Nrm(\textbf{a})}{2}} q^{\Nrm(\textbf{a})}.
\end{equation*}
To finish the proof, we note that there is a  unique ideal of norm $2$, namely $(\sqrt{2})$ of $O_K$     which gives 
Theorem \ref{2mod1thm}.  
\end{proof}
\begin{proof}[Proof of Theorem \ref{3mod4thm}] 
\begin{proposition} \label{Hecke5}
We have 
$$
f_5(q) = \sum_{\substack{n\geq 0\\ |j|\leq n}} q^{\frac{3(n^2+n)}{2} - \frac{j^2+j}{2}}.
$$
\end{proposition} 
\begin{proof} 
We use Theorem  \ref{AndrewBailey} with    $b=q^{1/2}$, $c=-q^{1/2}$, and $a=q$  in  Bailey's Lemma with   $\rho_1=q$ and $\rho_2\to \infty$. The ``$\beta$-side''   becomes
\begin{eqnarray*}
 \frac{1}{1-q}
  \sum_{n=0}^{\infty}  (-1)^n \frac{(q)_n}{(q^3;q^2)_{n}} q^{\frac{n^2+n}{2}}
  = f_5(q).
\end{eqnarray*}
The ``$\alpha$-side'' equals
$$
\sum_{n=0}^{\infty} 
q^{ \frac{3(n^2+n)}{2} }   
\left(
1+\sum_{j=1}^{n} q^{-\frac{j^2+j}{2}}(1+q^{j})\right),
$$
which is the right-hand side of Proposition \ref{Hecke5}. 
\end{proof}
It follows easily by Proposition \ref{Hecke5} that 
$$
q^2 f_5(q^8) = \sum_{\substack{n\geq 0\\ |j|\leq n}} q^{3 (2n+1)^2-(2j+1)^2}.
$$
Using Lemma 3 of \cite{ADH},  we 
write $\textbf{a}=(u +\sqrt{3}v)$ with $\Nrm(\textbf{a}) = 3v^2-u^2$, $v>0$, and $-v < u \leq v$. 
This yields as before 
$$
q^2 f_5(q^8) = \sum_{\substack{\textbf{a}\subset O_L\\ \Nrm(\textbf{a})\equiv 2\pmod{8}}} q^{\Nrm(\textbf{a})}.
$$
The theorem then follows after dividing by the unique ideal $(1+\sqrt{3})$ of $O_L$ of norm $2$.
\end{proof}
\begin{proof}[Proof of Theorem \ref{3mod4compthm}]  

\begin{proposition}\label{Hecke6}
We have 
$$
f_6(q) = 
-\sum_{\substack{n \geq 0 \\ |j| \leq n} } 
q^{3n^2+3n+1-j^2}
- \sum_{\substack{n \geq 0  \\-n \leq j \leq n-1 }} 
q^{3n^2-j^2-j}.
$$
\end{proposition}
\begin{proof}
We take the Bailey pair from Lemma 12 of \cite{An} with respect to $a=1$ with  
\begin{eqnarray*}
\mathcal{B}_0&=&\mathcal{B}_0(q):=0,\\
\mathcal{B}_n&=&\mathcal{B}_n(q):=\frac{1}{(q^n)_n}=\frac{(q)_{n-1}}{(q)_{2n-1}},
\end{eqnarray*} 
and 
\begin{eqnarray}
\label{Aeven}
\mathcal{A}_{2n}&=&\mathcal{A}_{2n}(q) :=-q^{3n^2-2n}(1-q^{4n})\sum_{j=-n}^{n-1}q^{-j^2-j},\\
\label{Aodd}
\mathcal{A}_{2n+1}&=&\mathcal{A}_{2n+1}(q) := q^{3n^2+n}(1-q^{4n+2})\sum_{|j|\leq n} q^{-j^2}.
\end{eqnarray}
  We take in Bailey's Lemma  
  $\rho_1=-1$, divide on both sides by $2(1-\rho_2)$, and then let $\rho_2\to 1$. 
On the ``$\beta$-side''  we then have
$$
\sum_{n=1}^{\infty} 
 (-1)^n 
 \frac{(-q)_{n-1}(q)_{n-1}^2}{(q)_{2n-1}}  
 = f_6(q).
$$
The ``$\alpha$-side'' equals  
$$
\sum_{n=1}^{\infty}
\frac{q^{2n}}{1-q^{4n}}\mathcal{A}_{2n}  
- \sum_{n \geq 0} \frac{q^{2n+1}}{1-q^{4n+2}}\mathcal{A}_{2n+1}.
$$
Inserting (\ref{Aeven}) and (\ref{Aodd}) gives  Proposition \ref{Hecke6}.
\end{proof}
From Proposition \ref{Hecke6} we have 
$$
q^{-1}f_6(q^4) = - \sum_{\substack{n\geq 0\\ |j| \leq n}}  q^{3(2n+1)^2 - (2j)^2} - \sum_{\substack{n\geq 0\\ -n\leq j\leq n-1}}  q^{3(2n)^2 - (2j+1)^2}.
$$
From this we immediately conclude the theorem as before.  
\end{proof}

\begin{proof}[Proof of Theorem \ref{3case}] 
\begin{proposition} \label{Hecke7}
We have 
$$
f_7(q)= \sum_{\substack{n \geq 0\\   |j| \leq n  }} (-1)^{n+j} 
q^{3n^2+2n-j^2}\left( 1-q^{2n+1}\right) .
$$
\end{proposition}
\begin{proof}
We use Theorem \ref{AndrewBailey} with $q \mapsto q^2, b=-1, c=-q,$ and $a=q^2$ in Bailey's Lemma with  $ \rho_1=q^2$ and $\rho_2 \to \infty$.
The ``$\beta$-side''  is 
$$
\frac{1}{1-q^2} \sum_{n=0}^{\infty}\frac{\left( q^2;q^2\right)_n (-1)^n q^{n(n+1)}}{\left(-q^2,-q^3;q^2 \right)_n} =\frac{1}{1-q} \sum_{n=0}^{\infty}\frac{\left( q^2;q^2\right)_n (-1)^n q^{n(n+1)}}{(-q)_{2n+1}} =\frac{f_7(q)}{1-q}.
$$
The ``$\alpha$-side''  equals the right hand side of Proposition \ref{Hecke7} divided by $1-q$, giving Proposition \ref{Hecke7}.
\end{proof}
From Proposition \ref{Hecke7} we immediately obtain
\begin{equation} \label{Hecke7new}
q f_7(q^3) =  \sum_{\substack{n \geq 0\\   |j| \leq n  }} (-1)^{n+j} \left(
q^{(3n+1)^2-3 j^2} - q^{(3n+2)^2 -3j^2}\right).
\end{equation}
We use Lemma 3 of \cite{ADH} to 
write $\textbf{a}= (u+\sqrt{3}v)$ with $\Nrm(\textbf{a})=u^2-3v^2$, $u>0$, $-\frac{u}{3}<v \leq \frac{u}{3}$. The condition $\Nrm(\textbf{a}) \equiv 1 \pmod 3$ translates into the two summands of (\ref{Hecke7new}), noting that  
$$
(-1)^{n+j}=-(-1)^{(3n+1)^2-3j^2} = (-1)^{(3n+2)^2 -3j^2}.
$$
\end{proof}

\begin{proof}[Proof of Theorem \ref{3mod3compthm}]
\begin{proposition}\label{Hecke8}
We have 
$$
f_8(q)=-\sum_{\substack{n\geq 1\\ -n\leq j\leq n-1}}q^{6n^2-2n-2j^2-2j}(1+q^{4n}) + \sum_{\substack{n\geq 0\\ |j|\leq n}} q^{6n^2+4n+1-2j^2}(1+q^{4n+2}).
$$
\end{proposition}
\begin{proof}
We use the Bailey pair $\mathcal{B}_n\left(q^2 \right)$ and $\mathcal{A}_n\left( q^2\right)$  
from the proof of Theorem \ref{3mod4compthm}.  
We then let in Bailey's Lemma $\rho_1=q$, divide both sides by $1-\rho_2$ and take $\rho_2\to 1$. 
The ``$\beta$-side'' gives 
$$
\sum_{n =1}^{\infty} 
 \frac{(q^2;q^2)_{n-1}(q;q^2)_{n}}{(q^{2n};q^2)_n} q^n = \sum_{n=1}^{\infty} 
  \frac{(q)_{2n-1}}{(q^{2n};q^2)_n} q^n = f_8(q).
$$
On the ``$\alpha$-side''  we have 
$$
\sum_{n\geq 1} \frac{q^n}{1-q^{2n}}\mathcal{A}_n(q^2).
$$
Proposition \ref{Hecke8} then follows after using (\ref{Aeven}) and (\ref{Aodd}) to evaluate $\mathcal{A}_n(q^2)$.

\end{proof}
By Proposition \ref{Hecke8}, $q^{-2}f_8(q^6)$ equals
$$
- \sum_{\substack{n\geq 0\\ -n-1\leq j\leq n}} \left(q^{(6n+7)^2-3(2j+1)^2}+q^{(6n+5)^2-3(2j+1)^2}\right) + \sum_{\substack{n\geq 0\\ |j|\leq n}} \left(q^{(6n+2)^2-3(2j)^2}+q^{(6n+4)^2-3(2j)^2}\right).
$$
As before we see that 
\begin{equation}
q^{-2}f_8(q^6)=
\sum_{\substack{  \textbf{a} \subset O_L \\ \mathcal{N}(\textbf{a})\equiv -2\pmod{6}   }} 
(-1)^{\frac{\Nrm(\textbf{a})}{2}} q^{\Nrm(\textbf{a})}.
\end{equation}   Theorem \ref{3mod3compthm} then follows by dividing by the unique ideal $(1+\sqrt{3})$ of $O_L$ of norm $2$.
\end{proof}
\subsection{Proof of the  Corollories} \label{CorSection}
 Corollary  \ref{LacCorollary} follows immediately from the Hecke-type sums obtained in Section \ref{ProofSection}, using Theorem 1 of \cite{Od}.  
 
 To prove Corollary \ref{InfCorollary}, we 
  note that in all cases the weightings  only depended on the norm of $\textbf{a}$, and hence the $n$-th coefficient in each case equals the number of ideals of norm $n$ times a weighting depending only on $n$.  We next recall 
that unique factorization gives the number of ideals of $O_K$ and $O_L$ of norm $p^{\ell}$ as follows:
\begin{eqnarray*}
\#\{ \textbf{a}\subset O_K: \Nrm(\textbf{a})=p^{\ell}\} &=& 
\left\{ 
\begin{array}{ll}
\ell+1& \text{if } p \equiv  \pm 1 \pmod 8,\\
1& \text{if } p \equiv  \pm 3 \pmod 8 \text{ and } \ell  \text{ is even },\\
1& \text{if }  p=2,\\
0& \text{if } p \equiv  \pm 3 \pmod 8 \text{ and } \ell  \text{ is odd },\\
\end{array}
\right. \\ 
\#\{ \textbf{a}\subset O_L: \Nrm(\textbf{a})=p^{\ell}\}
&=& 
\left\{ 
\begin{array}{ll}
\ell+1& \text{if } p \equiv  \pm 1 \pmod{12},\\
1& \text{if } p \equiv  \pm 5  \pmod{12} \text{ and } \ell  \text{ is even },   \\
1& \text{if }  p=2,3,\\
 0& \text{if } p \equiv \pm 5 \pmod{12}  \text{ and } \ell  \text{ is odd }.\\
\end{array}
\right. 
\end{eqnarray*}
Multiplicativity then gives an exact formula for the number of elements of norm $n$, based on the factorization. The proof of  Corollary \ref{InfCorollary}  now follows easily using the explicit form of the weighting.   

\section{Combinatorial Interpretations}\label{combinatorics}

\subsection{Theorem \ref{16Theorem}} \label{combinatorics1}
 We rewrite
 $$
 f_1(q) =\sum_{n=0}^{\infty} \frac{q^{ \frac{n^2+n}{2 }} }{(-q)_n \left(1-q^{2n+1} \right)}.
 $$
For $n \geq 1$ we   form a triangle with sides of length $n$ in the upper left hand corner of the Ferrer's diagram.  We then append $2k$ to the first $n-1$ parts and then add a part of size $2k$ and $k$ parts of size $1$, where $k$ is the power of $q^{2n+1}$ arising from $\frac{1}{1-q^{2n+1}}$.  Finally, $\frac{1}{(-q)_n}$ adds parts of size at most $n$, which we place along the ``diagonal''.  Note that $n$ counts the number of non-repeated parts and therefore our construction yields a bijection 
to the partitions in $P_1$ with at least one nonrepeated part. 
 The case $n=0$ corresponds to partitions containing only ones.  Subtracting the term $q$ removes the unique partition of $1$, which was already counted above.  From the construction it is clear that the weighting is given by the parity of $r_1(\lambda)$.  
\subsection{Theorem \ref{16Companion}} \label{combinatorics2}
We rewrite 
$$
f_2(q) = \sum_{n\geq 1} \frac{q^{\frac{n^2+n}{2}}}{(-q)_{n-1} (1-q^{2n-1})}.
$$
We begin by forming a triangle with sides of length $n$.  We then append $2k$ to the largest $n-1$ parts and append $k$ to the smallest part, where $k$ is the power of $q^{2n-1}$ coming from $\frac{1}{1-q^{2n-1}}$.  The smallest part is now $k+1$, the rank is $2k$, and the second smallest part is $2k+2$.  Finally, $\frac{1}{(-q)_{n-1}}$ adds parts of size at most $n-1$, which we place along the diagonal.  The size of the smallest part is not changed in the final step, so our construction yields a bijection to $\lambda\in P_2$ with $n$ parts and smallest part $k+1$, having the desired weighting.

\subsection{Theorem \ref{2twistthm}} \label{combinatorics3}
To see that $f_3(q)$ generates the desired partitions, we rewrite 
$$
f_3(q)=
\sum_{n= 0}^{\infty} 
 \frac{(q)_{n}(q^{n+1})_n }{(-q)_n (-q^{n+1})_{n+1}} q^n.
$$
The factor $\frac{(q)_n}{(-q)_n} q^n$ gives an overpartition  
(corresponding to the first  component $\mu$ of the overpartition pair $(\mu,\lambda)$)
 with largest part $n$ occuring at least once non-overlined, weighted by the number of parts of $\mu$ minus $1$.  
 We next consider the second component $\lambda$ of the overpartition pair, and  
  add   $k$  parts of size $n+1$ and $k$  parts of size $r$ (when $r>0$), 
 where $k$ is the power of $q^{n+1+r}$ coming from $\frac{(q^{n+1})_n}{(-q^{n+1})_{n+1}}$.  If a part of size $r>0$ comes from the numerator, then we overline the first occurrence of $r$ and if the term with $r=0$ is chosen from the numerator, then we overline the first occurrence of $n+1$.  
 Since the parts for $r>0$ come in pairs with one of the parts of size $n+1$, we get $\Lambda \in P_3$ and the weighting is clearly given as desired. 

\subsection{Theorem \ref{2mod1thm}} \label{combinatorics4}
To see that $f_4(q)$ generates the desired partitions, we proceed similarly as for $f_3$ and rewrite
 $$
 f_4(q)=
\sum_{n=1}^{\infty} 
 \frac{(q)_n}{(-q)_n}q^{n}\cdot \frac{(q^{n+1})_{n-1}}{(-q^{n+1})_{n}}.
$$
The factor $\frac{(q)_n}{(-q)_n} q^n$ gives an overpartition (corresponding to the component $\mu$)
 with largest part $n$ occurring at least once non-overlined, weighted by the number of parts minus 1.  
 We next consider the component $\lambda$. 
 We add $k$ parts of size $n$ and $k$ parts of size $r$, where $k$ is the power of $q^{n+r}$ coming from $\frac{(q^{n+1})_{n-1}}{(-q^{n+1})_{n}}$.  If a part of size $r$ comes from the numerator, then we overline the first occurrence of $r$.
 From this one easily sees that $f_4$ enumerates the claimed partitions. 

\subsection{Theorem \ref{3mod4thm}} \label{combinatorics5}
Recall that  
$\ell_i=\lambda_i-\lambda_{i+1}$ ($1\leq i \leq n-1$) and  $\ell_n=\lambda_n$,  
 $E:=E_{\lambda}:=\{ 2 \leq  r \leq n: \ell_r\text{ is even}\}$ and $e:=\# E$.   
 The second largest part of $\lambda$ (if it exists) must be at least $n-1 + e$ because the size of the second largest part is $\lambda_2=\sum_{i=2}^{n} \ell_i$. 
  Thus, it is natural to define 
$$
d_{\lambda,1}:=\lambda_2- \left( (n-1)+e\right).
$$
Since 
$$
d_{\lambda,1} = \left(\sum_{i=2}^{n} \ell_i\right) - (n-1+e) = \sum_{i\in E} (\ell_i -2)  + \sum_{i\notin E, i>1} (\ell_i -1)
$$
is even, we may furthermore define the integral metric $d_{\lambda}:=\frac{d_{\lambda,1}}{2}$ alluded to in the introduction.

We now rewrite 
$$
f_5(q)= \frac{1+q^2}{1-q^3} + \sum_{n=2}^{\infty} 
 (-1)^n \frac{(q^2)_{n-1}}{(q^3;q^2)_{n}} q^{\frac{n^2+n}{2}}.
$$
The term
$$
\frac{1+q^2}{1-q^3}=\sum_{j=0}^{\infty}  \left(q^{3j} + q^{3j+2}\right)
$$
corresponds  to partitions in $P_5$ with exactly one part which is not congruent to $1$ modulo $3$.

For $n\geq 2$, we first form a triangle with sides of length $n$ weighted by the largest part.  We adjoin $3k_r$ to the largest part and $2k_r$  to the next $r-1$ parts, where $k_r$ is the power of $q^{2r+1}$ coming from $\frac{1}{(q^3;q^2)_{n}}$.   At this stage of the construction we have added an even number to each part other than the largest, so $\ell_i$ is odd for $i>1$ and hence $E=\emptyset$.  Notice that $d_{\lambda}=\sum_{r=2}^{n} k_r$ and $\ell_1= 3k_1 + d_{\lambda}+1$.  Hence we have thusfar constructed all $\lambda\in P_5$ with $E=\emptyset$, $n$ parts, and the desired weighting.

Finally we add $1$ to the first $r$ parts if $q^{r}$ is chosen from $(q^2)_{n-1}$.
Notice that $n$ represents the number of parts and $r\in E$ if and only  if 
we chose $q^r$ in the final step.  
Moreover  the overall weighting is 
$$
(-1)^{n+e} = (-1)^{n+e+2d_{\lambda}} =- (-1)^{\lambda_2}.
$$

\subsection{Theorem \ref{3mod4compthm}} \label{combinatorics6}

We first rewrite
$$
f_6(q) = \sum_{n=1}^{\infty} 
 (-1)^n \frac{(q^2;q^2)_{n-1}}{(q^n)_{n}}q^n.
$$
We begin with a part of size $n$ from the factor $q^n$.  We next add $k$ parts of size $n$ and $k$ parts of size $r$ (if $r>0$), where $k$ is the power of $q^{n+r}$ coming from $\frac{1}{(q^n)_{n}}$.  We finally add two parts of size $r$ and overline the first occurrence of $r$ if $q^{2r}$ is chosen from $(q^2;q^2)_{n-1}$.  Since all of the parts other than the ones coming from  $\frac{1}{1-q^n}$ occur in pairs with one of the parts either overlined or equal to $n$ we obtain an overpartition in $P_6$ with the   correct weighting.

\subsection{Theorem \ref{3case}} \label{combinatorics7}

We begin by rewriting 
$$
f_7(q)  
= \frac{1}{1+q} + \sum_{n=1}^{\infty} 
(-1)^n\, \frac{q^{n^2+n}(q)_n}{(-q^{n+1})_{n+1} }.
$$
The term $\frac{1}{1+q}$ corresponds to partitions with all parts equal to one weighted by the number of parts.
For $n\geq 1$ we first place $n$ parts of size $n+1$.  We then append $k$ to the first $n$ parts and add $k$ parts of size $r$ weighted by $(-1)^k$, where $k$ is the power of $q^{n+r}$ chosen from $\frac{1}{(-q^{n+1})_{n+1}}$.  Finally, we add an overlined part of size $r$ weighted by $-1$ if $q^r$ is chosen from $(q)_n$.  Notice that   $n=M(\lambda)$.  Moreover, since the largest part is at least $2$ we have no overlap with the partitions coming from $\frac{1}{1+q}$.  
This easily yields the desired bijection to $P_7$ with the claimed weighting.

\subsection{Theorem \ref{3mod3compthm}} \label{combinatorics8}
We start with a part of size $n$.  For $q^r$ chosen from $(q)_{n-1}$ we add an overlined part of size $r$.  We finally add $k$ parts of size $n$ and $k$ parts of size $r$ (when $r>0$), where $k$ is the power of $q^{n+r}$ occurring from $\frac{1}{(-q^n)_n}$.   
This easily yields the desired bijection to $P_8$ with the claimed weighting.

\end{document}